\documentclass[a4paper,11pt,oneside,reqno]{amsart}
\usepackage{style}
\title{Convergence rate of Euler-Maruyama scheme for McKean-Vlasov SDEs with density-dependent drift}
\author[Le]{Anh-Dung Le}
\address{Toulouse School of Economics, 1 Esplanade de l'Université, 31000 Toulouse, France}
\email[Le]{leanhdung1994@gmail.com}
\date{\today}
\thanks{This work was supported by ANR MaSDOL (No. ANR-19-CE23-0017); Air Force Office of Scientific Research, Air Force Material Command, USAF (No. FA9550-19-7026); and ANR Chess (No. ANR-17-EURE-0010).}

\begin{document}

\begin{abstract}
In this paper, we study weak well-posedness of a McKean-Vlasov stochastic differential equations (SDEs) whose drift is density-dependent and whose diffusion is constant. The existence part is due to Hölder stability estimates of the associated Euler-Maruyama scheme. The uniqueness part is due to that of the associated Fokker-Planck equation. We also obtain convergence rate in weighted $L^1$ norm for the Euler-Maruyama scheme.
\end{abstract}

\maketitle



\tableofcontents

\section{Introduction} \label{intro}
We fix $p \in [1, \infty)$ and let $\sP_p (\bR^d)$ be the space of Borel probability measures on $\bR^d$ with finite $p$-th moment. We endow $\sP_p (\bR^d)$ with Wasserstein metric $W_p$. We fix $T \in (0, \infty)$ and let $\bT$ be the interval $[0, T]$. We consider a measurable function
\[
b : \bT \times \bR^d \times \bR_+ \times \sP_p (\bR^d) \to \bR^d.
\]

Let $(B_t, t \ge 0)$ be a $d$-dimensional Brownian motion and $\bF \coloneq (\cF_t, t \ge 0)$ be an admissible filtration on a probability space $(\Omega, \mathcal A, \bP)$. We assume that $(\Omega, \cA, \bF, \bP)$ satisfies the usual conditions. In this paper, we consider the SDE
\begin{equation} \label{main_eq1}
\diff X_t = b(t, X_t, \ell_t (X_t), \mu_t) \diff t + \sqrt 2 \diff B_t,
\quad t \in \bT,
\end{equation}
where the marginal distribution of $X_0$ is $\nu$, that of $X_t$ is $\mu_t$, and the marginal density of $X_t$ is $\ell_t$. The term ``McKean-Vlasov'' is due to the presence of $\mu_t$ while the term ``density-dependent'' is due to that of $\ell_t (x)$.

To show the existence of a weak solution, we consider an Euler-Maruyama scheme $X^n \coloneq (X^n_t, t \in \bT)$ constructed as follows. Let $n \in \bN$ and $\eps_n \coloneq T/n$. Let $t_k \coloneq k \eps_n$ for $k \in [\![0, n ]\!] \coloneq \{0, 1, \ldots, n\}$. Let $\tau^n_t \coloneq t_k$ if $t \in [t_k, t_{k+1})$ for some $k \in [\![0, n-1 ]\!]$. For every $t \in \bT$, let
\begin{equation} \label{el_sch}
X^n_t \coloneq X_0 +  \int_0^t b (s, X^n_{\tau^n_s}, \ell^n_{\tau^n_s} (X^n_{\tau^n_s}), \mu^n_{\tau^n_s}) 1_{(\eps_n, T]} (s) \diff s + \sqrt 2 B_t ,
\end{equation}
where $\ell^n_s, \mu^n_s$ are the density and distribution of $X^n_s$ respectively.

We consider the following set of assumptions:
\begin{enumerate}[label=\textbf{(A)}]
\item \label{main_assmpt1}
There exist constants $\alpha \in (0, 1), p \in [1, \infty), C > 0$ such that for all $t \in \bT; x, y \in \bR^d; r, r' \in \bR_+$ and $\varrho, \varrho' \in \sP_p (\bR^d)$:
\begin{enumerate}[label=(A\arabic*)]
\item \label{main_assmpt1:1} $| b (t, x, r, \varrho) | \le C$.
\item \label{main_assmpt1:2} $\nu \in \sP_{p+\alpha} (\bR^d)$ has a density $\ell_\nu \in C^{\alpha}_b (\bR^d)$.
\item \label{main_assmpt1:3} $| b (t, x, r, \varrho) - b (t, x, r', \rho') | \le C \{ |r- r'| + W_p (\varrho, \varrho') \}$.
\end{enumerate}
\end{enumerate}

We require $\nu \in \sP_{p+\alpha} (\bR^d)$ and not just $\nu \in \sP_{p} (\bR^d)$. By Markov's inequality, this gives us a direct control \eqref{main_thm2:ineq1} of the tail of $p$-moment. We gather parameters in \ref{main_assmpt1}:
\[
\Theta_1 \coloneq (d, T, \alpha, p, C).
\]

First, we prove the following estimates about Hölder regularity:

\begin{theorem} \label{main_thm1}
Let \ref{main_assmpt1} hold. There exist constants $c_1>0$ (depending on $\Theta_1$) and $c_2>0$ (depending on $\Theta_1, \nu$) such that for all $0 \le s < t \le T$:
\begin{align}
\sup_n \sup_{t \in \bT} \| \ell^n_t \|_{C^\alpha_b} &\le c_1 \| \ell_\nu \|_{C^\alpha_b}, \label{main_thm1:ineq1} \\
\sup_n \| \ell^n_t - \ell^n_s \|_\infty &\le c_2 (t-s)^{\frac{\alpha}{2}}, \label{main_thm1:ineq2} \\
\sup_n W_p (\mu^n_t, \mu^n_s) &\le c_1 (t-s)^{\frac{1}{2}}, \label{main_thm1:ineq3} \\
\sup_n \int_{\bR^d} (1+|x|^p) |\ell^n_t (x) - \ell^n_s (x)| \diff x &\le c_2 (t-s)^{\frac{\alpha}{2}} s^{-\frac{\alpha}{2}}. \label{main_thm1:ineq4} 
\end{align}

If we assume, in addition, that $\int_{\bR^d} (1+|x|^p) \sqrt{\ell_\nu (x)} \diff x  \le C$, then
\begin{equation} \label{main_thm1:ineq5}
\sup_n \int_{\bR^d} (1+|x|^p) |\ell^n_t (x) - \ell^n_s (x)| \diff x \le c_2 (t-s)^{\frac{\alpha}{4}}.
\end{equation}
\end{theorem}

Estimate \eqref{main_thm1:ineq4} will be used in the proof of \zcref{main_thm3}, but it explodes as $s \downarrow 0$. As in \cite{song_convergence_2024}, we use time cutoff $1_{(\eps_n, T]}$ in \eqref{el_sch} to tackle this issue. As a direct application of \zcref{main_thm1}, we obtain

\begin{theorem} \label{main_thm2}
Let \ref{main_assmpt1} hold. Then \eqref{main_eq1} has a weak solution whose marginal density satisfies the estimates in \zcref{main_thm1}.
\end{theorem}

The main result of the paper is the following rate of convergence:

\begin{theorem} \label{main_thm3}
Let \ref{main_assmpt1} hold and $p=1$. There exists a constant $c>0$ (depending on $\Theta_1, \nu$) such that for any $n \in \bN$:
\[
\sup_{t \in \bT} \int_{\bR^d} (1+|x|)|\ell^n_t (x) - \ell_t(x)| \diff x \le c n^{-\frac{\alpha}{2}}.
\]
\end{theorem}

Next we review works related to \eqref{main_eq1} which is a special case of the SDE
\begin{equation} \label{main_eq3}
\diff X_t = b(t, X_t, \ell_t (X_t), \mu_t) \diff t + \sigma (t, X_t, \ell_t (X_t), \mu_t) \diff B_t,
\quad t \in \bT,
\end{equation}
for some measurable functions
\begin{align}
b &: \bT \times \bR^d \times \bR_+ \times \sP_p (\bR^d) \to \bR^d , \\
\sigma &: \bT \times \bR^d \times \bR_+ \times \sP_p (\bR^d) \to \bR^d \otimes \bR^d .
\end{align}

Equation \eqref{main_eq3} is usually studied in two separate cases:

\begin{enumerate}
\item The first case is where $b(t, x, r, \mu) = b(t, x, \mu)$ and $\sigma (t, x, r, \mu) = \sigma (t, x, \mu)$, i.e., there is only distribution-dependence. Then \eqref{main_eq3} is commonly called (classical) McKean-Vlasov SDEs and appears as the limit of various interacting particle systems in biology (see e.g. \cite{naldi2010mathematical,muntean2014collective,degond2018mathematical}), mean-filed games (see e.g. \cite{carmona2018probabilistic,cardaliaguet2019master}), and data science (see e.g. \cite{mei2018mean,sirignano2020mean,rotskoff2022trainability}). This form of \eqref{main_eq3} has been extensively studied in various aspects, including well-posedness (see e.g. \cite{mishura2020existence,huang2019distribution,rockner_well-posedness_2021}), derivative formulas (see e.g. \cite{huang_derivative_2021,wang_derivative_2023}), long-time behavior (see e.g. \cite{liu_existence_2022,ren_exponential_2021,ren_singular_2023a}), and propagation of chaos (see e.g. \cite{lacker_hierarchies_2023,sznitman1991topics,erny_well-posedness_2022}).
\item The second case is where $b(t, x, r, \mu) = b(t, x, r)$ and $\sigma (t, x, r, \mu) = \sigma (t, x, r)$, i.e., there is only density-dependence. Examples of this form are Burgers' equations \cite{calderoni_propagation_1983,gutkin1983propagation,osada1985propagation,sznitman_propagation_1986} and (generalized) porous media equations \cite{inoue1991derivation,benachour1996processus,blanchard_probabilistic_2010,barbu_probabilistic_2011,belaribi_uniqueness_2012,belaribi_probabilistic_2013}. Recent works on this form include \cite{barbu2018probabilistic,cohen_wellposedness_2019,barbu2020nonlinear,barbu2021solutions,barbu2021uniqueness,barbu_evolution_2023,barbu_uniqueness_2023}. If $\sigma$ is density-dependent, then the map $r \mapsto r(\sigma \sigma^\top) (t, x, r)$ is usually assumed to satisfy some monotonic condition. For Euler-Maruyama scheme in the special case of constant-diffusion, we refer to \cite{hao2021euler,hao_phd_2023,wu_well-posedness_2023,song_convergence_2024}.
\end{enumerate}

There appears to be a shortage of results about SDEs that contain both distribution- and density-dependence. In this direction, we have found only four papers \cite{wang2023singular,hao_strong_2024,zhang_second_2022,le2024wellposedness}. This scarcity is the motivation for our paper. To our best knowledge, the current paper is the first to study Euler-Maruyama scheme of an SDE that contains both distribution- and density-dependent. While \cite{hao2021euler,hao_phd_2023,wu_well-posedness_2023,song_convergence_2024} obtained convergence rate of marginal density in $L^1$ norm, we obtain that in (stronger) weighted $L^1$ norm where the weight is $x \mapsto 1 + |x|$.

The paper is organized as follows. In \zcref{preliminary}, we recall definition of Wasserstein space and heat kernel estimates; prove some crude estimates about the scheme; and establish Duhamel representation of marginal density. We prove our results in \zcref{proof:main_thm1}, \zcref{proof:main_thm2}, and \zcref{proof:main_thm3}, respectively.

We recall two notions of a solution of \eqref{main_eq1}:

\begin{definition}~
\begin{enumerate}
\item A \textit{strong solution} to \eqref{main_eq1} is a continuous $\bR^d$-valued $\bF$-adapted process $(X_t, t \in \bT)$ on $(\Omega, \cA, \bP)$  such that for each $t \in \bT$: the distribution of $X_t$ is $\mu_t \in \sP_p (\bR^d)$ and admits a density $\ell_t$; and
\begin{equation} \label{sol_def1}
X_t = X_0 + \int_0^t b(s, X_s, \ell_s (X_s), \mu_s) \diff s + B_t
\qtext{$\bP$-a.s.}
\end{equation}

\item A \textit{weak solution} to \eqref{main_eq1} is a continuous $\bR^d$-valued process $(X_t, t \in \bT)$ on some probability space $(\Omega, \cA, \bP)$ where there exist some $d$-dimensional Brownian motion $(B_t, t \ge 0)$ and some admissible filtration $\bF \coloneq (\cF_t, t \ge 0)$ such that $(X_t, t \in \bT)$ is $\bF$-adapted, its marginal distribution admits a density, and \eqref{sol_def1} holds.

\item Equation \eqref{main_eq1} has \textit{strong uniqueness} if any two strong solutions on a given probability space $(\Omega, \cA, \bP)$ for a given Brownian motion and a given admissible filtration coincide $\bP$-a.s. on the path space $C(\bT; \bR^d)$. Equation \eqref{main_eq1} has \textit{weak uniqueness} if any weak solution induces the same distribution on $C(\bT; \bR^d)$.
\end{enumerate}
\end{definition}

Throughout the paper, we use the following conventions:
\begin{enumerate}
\item We denote by $\nabla, \nabla^2, \Delta$ the gradient, the Hessian and the Laplacian with respect to (w.r.t) the spatial variable. We denote by $\partial_t$ (or $\partial_s$) the derivative w.r.t time. For $\alpha \in (0, 1]$ and a function $f : \bR^d \to \bR$, let $[f]_\alpha \coloneq \sup_{x \neq y} \frac{|f (x)-f (y)|}{|x-y|^\alpha}$ be the $\alpha$-Hölder constant of $f$. For $\alpha \in (0, 1)$, let $C^{\alpha}_b (\bR^d)$ be the Hölder space of functions $f : \bR^d \to \bR$ with the norm $\| f \|_{C^{\alpha}_b} \coloneq \|f\|_{\infty} + [f]_\alpha$.

\item Let $L^0 (\bR^d)$ be the space of real-valued measurable functions on $\bR^d$. Let $L^0_+ (\bR^d)$ be the subset of $L^0 (\bR^d)$ that consists of non-negative functions. Let $L^0_b (\bR^d)$ be the subset of $L^0 (\bR^d)$ that consists of bounded functions. We denote by $I_d$ the identity matrix in $\bR^d \otimes \bR^d$. We denote by $\lceil r \rceil$ the smallest integer greater than or equal to  $r \in \bR$.

\item For brevity, we write $\infty$ for $+\infty$. Let $\bN$ be the set of strictly positive natural numbers. Let $\bR_+ \coloneq \{x \in \bR : x \ge 0\}$. For natural numbers $m, n$ with $m \le n$, we denote $[\![m, n ]\!] \coloneq \{m, m+1, \ldots, n \}$. We denote by $\cB (\bR^d)$ the Borel $\sigma$-algebra on $\bR^d$. Let $\sP (\bR^d)$ be the space of Borel probability measures on $\bR^d$. We denote by $M_p (\varrho)$ the $p$-th moment of $\varrho \in \sP (\bR^d)$, i.e.,
\[
M_p (\varrho) \coloneq \int_{\bR^d} |x|^p \diff \varrho (x).
\]
\end{enumerate}

\section{Preliminaries} \label{preliminary}

\subsection{Wasserstein space $(\sP_p (\bR^d), W_p)$}

Let $\mu, \nu \in \sP (\bR^d)$. The set of \textit{transport plans} (or \textit{couplings}) between $\mu$ and $\nu$ is defined by
\[
\Gamma (\mu, \nu) \coloneq \{ \varrho \in \sP ({\bR}^d \times {\bR}^d ) : \mu = \pi^1_\sharp \varrho \text{ and } \nu = \pi^2_\sharp \varrho \} .
\]

Above, $\pi^i$ is the projection of $\bR^d \times \bR^d$ onto its $i$-th coordinate, and $\pi^i_\sharp \varrho \in \sP (\bR^d)$ is the push-forward measure of $\varrho$ through $\pi^i$ , i.e., $\{ \pi^i_\sharp \varrho \} (A) = \varrho (\{ \pi^i \}^{-1} (A))$ for all $A \in \cB (\bR^d)$. For $\mu, \nu \in \sP_p (\bR^d)$, we define
\[
W_p (\mu, \nu) \coloneq \inf_{\varrho \in \Gamma (\mu, \nu)} \left \{ \int_{\bR^d} |x-y|^p \diff \varrho (x, y) \right \}^{1/p}.
\]

By \cite[Theorem 6.18]{villani2009optimal}, $(\sP_p (\bR^d), W_p)$ is a Polish space. For more information about Wasserstein space, we refer to \cite{villani2003topics,villani2009optimal,Ambrosio2013,santambrogio2015optimal,figalli_invitation_2021,ambrosio2021lectures,maggi2023optimal}.

\subsection{The scheme is well-defined}

First, we recall the freezing lemma. For the sake of completeness, we include its proof in the Appendix.

\begin{lemma} \label{freezing}
Let $(\Omega, \cA, \bP)$ be a probability space and $\cD, \cG$ independent sub-$\sigma$-algebras of $\cA$. Let $(E, \cE)$ be a measurable space and $Y: \Omega \to E$ measurable w.r.t $\cD$. Assume that $\varphi:E \times \Omega \to \bR^d$ is measurable w.r.t $\cE \otimes \cG$ and that $\varphi (Y, \cdot)$ is integrable.
\begin{enumerate}
\item  Let $N \coloneq \{x \in E : \varphi(x, \cdot) \text{ not integrable}\}$. Let $\mu$ be the distribution of $Y$ on $(E, \cE)$. Then $N \in \cE$ and $\mu(N)=0$.
\item We define $\Phi : E \to \bR^d$ by $\Phi (y) \coloneq 0$ for $y \in N$ and $\Phi (y) \coloneq \bE [\varphi(y, \cdot)]$ for $y \in E \setminus N$. Then $\Phi$ is measurable and $\bE [\varphi(Y, \cdot) | \cD] = \Phi (Y)$.
\end{enumerate}
\end{lemma}

We define $b^n : \bT \times \bR^d \to \bR^d$ by
\begin{equation} \label{drift}
b^n (t, x) \coloneq b (t, x, \ell^n_{\tau^n_t} (x), \mu^n_{\tau^n_t}) 1_{(\eps_n, T]} (t).
\end{equation}

Then $X^n$ satisfies the SDE
\begin{equation} \label{main_eq2}
\diff X^n_t = b^n (t, X^n_{\tau^n_t}) \diff t +  \sqrt 2 \diff B_t.
\end{equation}

Second, we verify that \eqref{el_sch} is well-defined.

\begin{lemma} \label{well-defined}
Let \ref{main_assmpt1} hold. Then each $X^n_t$ in \eqref{el_sch} has distribution $\mu^n_t \in \sP_p (\bR^d)$ and admits a density.
\end{lemma}

\begin{proof}
We fix $t \in (t_k, t_{k+1}]$ for some $k \in [\![0, n -1]\!]$. By induction argument, we assume WLOG that the statement holds for $t_k$. We have
\[
X^n_t = X^n_{t_k} + \int_{t_k}^t b^n (s, X^n_{t_k}) \diff s + \sqrt 2 (B_t-B_{t_k}).
\]

It follows from $M_p (\mu^n_{t_k}) + \|b^n\|_\infty < \infty$ that $M_p (\mu^n_t) < \infty$. It remains to prove that $X^n_t$ admits a density. Let $\Sigma^1_t$ be the $\sigma$-algebra generated by $(B_s- B_r, t_k \le r < s \le t)$. Let $\Sigma^2_t$ be the $\sigma$-algebra generated by $X^n_{t_k}$. We define $F_t : \bR^d \times \Omega \to \bR^d$ by
\[
F_t (x, \cdot) \coloneq x+ \int_{t_k}^t b^n (s, x) \diff s + \sqrt 2 (B_t-B_{t_k}).
\]

Because $b^n$ is bounded, $F_t$ is well-defined and measurable w.r.t $\cB (\bR^d) \otimes \Sigma_t^1$. First, $X^n_t = F_t(X^n_{t_k}, \cdot)$. Second, $F_t(x, \cdot)$ has a non-degenerate normal distribution. We fix a Lebesgue-null set $A \in \cB (\bR^d)$. We need to prove $\bP [X^n_t \in A] = 0$. We define $\bar F_t : \bR^d \to \bR^d$ by $\bar F_t (x) \coloneq \bP [ F_t(x, \cdot) \in A ]$. Then $\bar F_t =0$. Because $\bF$ is admissible, $\Sigma^1_t$ is independent of $\Sigma^2_t$. We have
\begin{align}
\bP [X^n_t \in A] &= \bP [ F_t(X^n_{t_k}, \cdot) \in A ] \\
&= \bE [ \bP [ F_t(X^n_{t_k}, \cdot) \in A | \Sigma^2_t ] ] \\
&= \bE [\bar F_t (X^n_{t_k})] \qtext{by \zcref{freezing}}\\
&=0. 
\end{align}

This completes the proof.
\end{proof}

\subsection{Heat kernel estimates} \label{heat_kernel_estimate}

We consider the standard Gaussian heat kernel defined for all $t>0$ and $x \in \bR^d$ by
\begin{equation} \label{gaussian_heat_kernel:eq1}
p_t (x) \coloneq \frac{1}{(4 \pi t)^{\frac{d}{2}}} \exp \left ( -\frac{|x|^2}{4 t} \right ) .
\end{equation}

Then
\begin{align}
p_t (x) &\le 2^{\frac{d}{2}} p_{2t} (x), \label{gaussian_heat_kernel:ineq2} \\ 
p_t (x+y) &\le 2^{\frac{d}{2}} \exp \left ( \frac{|y|^2}{4 t} \right ) p_{2t} (x). \label{gaussian_heat_kernel:ineq1}
\end{align}

The following estimates are classical. See e.g. \cite[Lemma 2.1]{hao2021euler} for their proofs.

\begin{lemma} \label{gaussian_estimate}
\begin{enumerate}
\item There exists a constant $c>0$ (depending on $d$) such that for all $t >0$ and $x,y \in \bR^d$:
\begin{equation} \label{gaussian_heat_kernel:ineq3}
|\nabla p_{t} (x)| \le c t^{-\frac{1}{2}} p_{t} (x).
\end{equation}

\item For $\alpha \in (0, 1)$, there exists a constant $c>0$ (depending on $d, T, \alpha$) such that for all $0 <t \le T; x,y \in \bR^d$ and $i \in \{0, 1\}$:
\begin{equation} \label{gaussian_heat_kernel:ineq4}
|\nabla^i p_{t} (x) - \nabla^i p_{t} (y)| \le c |x-y|^\alpha t^{-\frac{i+\alpha}{2}} \{ p_{4t} (x) + p_{4t} (y) \}.
\end{equation}

\item For $\alpha \in (0, 1)$, there exists a constant $c>0$ (depending on $d, T, \alpha$) such that for all $0 \le s <t \le T; x \in \bR^d$ and $i \in \{0, 1\}$:
\begin{equation} \label{gaussian_heat_kernel:ineq5}
|\nabla^i p_{t} (x) - \nabla^i p_{s} (x)| \le c |t-s|^{\frac{\alpha}{2}} \{ t^{-\frac{i+\alpha}{2}} p_{2t} (x) + s^{-\frac{i+\alpha}{2}} p_{2s} (x) \}.
\end{equation}
\end{enumerate}
\end{lemma}

The associated semigroup $(P_{t}, t>0)$ is defined for all $x \in \bR^d$ and $f \in L^0_+ (\bR^d) \cup L^0_b (\bR^d)$ by
\begin{equation} \label{semi_group1}
P_{t} f (x) \coloneq \int_{\bR^d} p_{t} (x- y) f(y) \diff y.
\end{equation}

Finally, we prove some crude estimates.

\begin{lemma} \label{sup_bound1}
Let \ref{main_assmpt1} hold. Let $t \in (t_k, t_{k+1}]$ for some $k \in [\![0, n -1]\!]$. There exists a constant $c \ge 1$ (depending on $\Theta_1$) such that for every $\varphi \in L^0_+(\bR^d \times \bR^d) \cup L^0_b(\bR^d \times \bR^d)$:
\begin{align}
\bE [\varphi (X^n_{t_k}, X^n_t)] & = \int_{\bR^d \times \bR^d} \varphi(x, y) \ell^n_{t_k} (x) p_{t-t_k} \bigg (x-y + \int_{t_k}^t b^n (s, x) \diff s \bigg ) \diff x \diff y \label{sup_bound1:eq1} \\
& \le c \int_{\bR^d \times \bR^d} \varphi (x, y) \ell^n_{t_k} (x) p_{2(t-t_k)} (x-y) \diff x \diff y,  \label{sup_bound1:ineq1} \\
\| \ell^n_{t} \|_\infty &\le c \| \ell^n_{t_k} \|_\infty. \label{sup_bound1:ineq2}
\end{align}
\end{lemma}

\begin{proof}
Let $\Sigma^1_t, \Sigma^{2}_t$ and $F_t$ be defined as in the proof of \zcref{well-defined}. We have
\begin{align}
\bE [\varphi (X^n_{t_k}, X^n_t)] &= \bE [ \varphi(X^n_{t_k}, F_t(X^n_{t_k}, \cdot)) ] \\
&= \bE [ \bE [ \varphi(X^n_{t_k}, F_t(X^n_{t_k}, \cdot)) | \Sigma^2_t ]]. \label{density_sup_bound1:prf:eq1}
\end{align}

We define $\hat F_t : \bR^d \to \bR^d$ by $\hat F_t (x) \coloneq \bE [ \varphi(x, F_t(x, \cdot)) ]$. Notice that $F_t(x, \cdot)$ has a normal distribution with mean $x + \int_{t_k}^t b^n (s, x) \diff s$ and covariance matrix $2tI_d$. Consequently, the density of $F_t(x, \cdot)$ is $y \mapsto p_{t-t_k} (x - y + \int_{t_k}^t b^n (s, x) \diff s)$. There exists a constant $c_1\ge 1$ (depending on $\Theta_1$) such that
\begin{align}
\hat F_t (x) &= \int_{\bR^d} \varphi(x, y) p_{t-t_k} \bigg (x-y + \int_{t_k}^t b^n (s, x) \diff s\bigg ) \diff y \label{sup_bound1:prf:eq2}\\
&\le c_1 \int_{\bR^d} \varphi(x, y) p_{2(t-t_k)} (x -y) \diff y \qtext{by \eqref{gaussian_heat_kernel:ineq1} and \ref{main_assmpt1:1}}. \label{sup_bound1:prf:ineq1}
\end{align}

WLOG, we assume $\varphi \in L^0_b (\bR^d \times \bR^d)$. Then
\begin{align}
\bE [\varphi (X^n_{t_k}, X^n_t)] &= \bE [ \hat F_t (X^n_{t_k})] \qtext{by \eqref{density_sup_bound1:prf:eq1} and \zcref{freezing}} \\
&= \int_{\bR^d} \ell^n_{t_k} (x) \hat F_t (x) \diff x \label{sup_bound1:prf:eq3}\\
&\le c_1 \int_{\bR^d \times \bR^d} \varphi(x, y) \ell^n_{t_k} (x) p_{2(t-t_k)} (x -y) \diff x \diff y \qtext{by \eqref{sup_bound1:prf:ineq1}} \label{sup_bound1:prf:ineq2} .
\end{align}

Thus \eqref{sup_bound1:ineq1} follows. Clearly, \eqref{sup_bound1:prf:eq2} and \eqref{sup_bound1:prf:eq3} imply \eqref{sup_bound1:eq1}. The case $\varphi \in L^0_+(\bR^d \times \bR^d)$ follows from a truncated argument and monotone convergence theorem. For every $f \in L^0_+ (\bR^d)$,
\begin{align}
\bE [f(X^n_t)] & \le c_1 \int_{\bR^d \times \bR^d} f(y) \ell^n_{t_k} (x) p_{2(t-t_k)} (x -y) \diff x \diff y \qtext{by \eqref{sup_bound1:prf:ineq2}} \\
& \le c_1 \| \ell^n_{t_k} \|_\infty \int_{\bR^d \times \bR^d} f(y) p_{2(t-t_k)} (x -y) \diff x \diff y \\
& = c_1 \| \ell^n_{t_k} \|_\infty \| f\|_{L^1} .
\end{align}

By duality, $\| \ell^n_t \|_\infty \le c_1 \| \ell^n_{t_k} \|_\infty$. This completes the proof.
\end{proof}

\subsection{Duhamel representation}

\begin{lemma} \label{duhamel}
Let \ref{main_assmpt1} hold. It holds for every $x \in \bR^d$ that
\begin{align}
\ell^n_t (x) &= P_{t} \ell_\nu (x) + \int_0^t \bE [ \langle b^n (s, X^n_{\tau^n_s}), \nabla p_{t-s} (X^n_s- x) \rangle ] \diff s, \label{duhamel:eq1} \\
\ell_t (x) &=P_{t} \ell_\nu (x) + \int_0^t \bE [ \langle b (s, X_s, \ell_s (X_s), \mu_s), \nabla p_{t-s} (X_s- x) \rangle ] \diff s. \label{duhamel:eq2}
\end{align}
\end{lemma}

\begin{proof}
Let's prove \eqref{duhamel:eq1}. We fix $t \in (0, T]$ and $f \in C^\infty_c (\bR^d)$. We define
\[
g:[0, t] \times {\bR}^d \to \bR, \, (s, y) \mapsto (P_{t-s} f) (y).
\]

Then $(\partial_s + \Delta) g=0$. We have
\begin{align}
\diff g (s, X^n_s) &= \begin{multlined}[t]
\langle b^n (s, X^n_{\tau^n_s}), \nabla g (s, X^n_s) \rangle \diff s \\
+ (\partial_s + \Delta) g (s, X^n_s) \diff s + \diff M_s 
\end{multlined} \qtext{by Itô's lemma} \\
&= \langle b^n (s, X^n_{\tau^n_s}), \nabla g (s, X^n_s) \rangle \diff s + \diff M_s .
\end{align}

Above, $M_0 = 0$ and $\diff M_s = \{ \nabla g (s, X^n_s)\}^\top \diff B_s$. Then
\[
f (X^n_t) = g (0, X_0) + \int_0^t \langle b^n (s, X^n_{\tau^n_s}), \mu^n_{\tau^n_s}), \nabla g  (s, X^n_s) \rangle \diff s + M_t.
\]

To apply Fubini's theorem, we next verify $\sup_{s \in \bT} \| \nabla g (s, \cdot) \|_\infty < \infty$. It suffices to prove that the Lipschitz constant of $g (s, \cdot)$ is bounded uniformly in time. We consider the Gaussian process governed by the SDE
\[
\diff Y^{x}_{s,t} = \sqrt 2 \diff B_t,
\quad t \in [s, T], Y^{x}_{s, s}=x .
\]

We have
\begin{align}
| g (s, x) - g (s, y) | &= |\bE [f(Y^{x}_{s, t})] - \bE [f(Y^{y}_{s, t})]| \\
&\le \| \nabla f \|_{\infty} \bE [ |Y^{x}_{s, t} - Y^{y}_{s, t}| ] \\
&\le c_1 \| \nabla f \|_{\infty} |x-y| .
\end{align}

Above, the constant $c_1 >0$ is given by \cite[Theorem 1.1(2)]{huang2022singular}. By Fubini's theorem,
\begin{align}
\int_{\bR^d} f (x) \ell^n_t (x) \diff x &= \bE[g (0, X_0)] + \int_0^t \bE[\langle b^n (s, X^n_{\tau^n_s}), \nabla g  (s, X^n_s) \rangle ] \diff s \\
&\eqcolon I_1 + I_2.
\end{align}

By Leibniz integral rule,
\[
\nabla g (s, y) = \nabla_y \int_{\bR^d} p_{t-s} (y - x) f(x) \diff x = \int_{\bR^d} \nabla p_{t-s} (y - x) f(x) \diff x.
\] 

We have
\begin{align}
I_1 &= \int_{\bR^d} \ell_{\nu} (y) \{ P_{t} f \} (y) \diff y \\
&= \int_{\bR^d} f (x) \left ( \int_{\bR^d} p_{t} (y-x) \ell_{\nu} (y) \diff y \right ) \diff x \\
&= \int_{\bR^d} f (x) \{ P_{t} \ell_\nu \} (x) \diff x, \\
I_2 &= \int_0^t \bE [ \langle b^n (s, X^n_{\tau^n_s}), \int_{\bR^d} f(x) \nabla p_{t-s} (X^n_s - x)  \diff x \rangle ] \diff s \\
&= \int_{\bR^d} f(x) \int_0^t \bE [ \langle b^n (s, X^n_{\tau^n_s}), \nabla p_{t-s} (X^n_s - x)  \rangle ] \diff s \diff x.
\end{align}

Then \eqref{duhamel:eq1} follows. The proof of \eqref{duhamel:eq2} is a straightforward modification of above reasoning.
\end{proof}

For brevity, we adopt the following notation in the remaining of the paper:
\begin{enumerate}
\item We write $M_1 \lesssim M_2$ if there exists a constant $c>0$ (depending on $\Theta_1$) such that $M_1 \le c M_2$.
\item We write $M_1 \preccurlyeq M_2$ if there exists a constant $c>0$ (depending on $\Theta_1, \nu$) such that $M_1 \le c M_2$.
\end{enumerate}

\section{Stability estimates of the scheme} \label{proof:main_thm1}

This section is devoted to the proof of \zcref{main_thm1}. We utilize techniques from \cite{wang2023singular,hao2021euler}. By \eqref{el_sch},
\begin{equation}
X^n_t - X^n_s  = \int_s^t b^n (r, X^n_{\tau^n_r}) \diff r + \sqrt 2 (B_t-B_s) .
\end{equation}

Then
\[
\bE [ |X^n_t - X^n_s|^p] \lesssim (t-s)^p + \bE [|B_{t-s}|^p] \lesssim (t-s)^{\frac{p}{2}} .
\]

Thus \eqref{main_thm1:ineq3} follows. 

\subsection{Bound supremum norm of marginal density}

By \zcref{duhamel} and \ref{main_assmpt1:1}, there exists a constant $c_1 >0$ (depending on $\Theta_1$) such that
\begin{align}
\ell^n_{t} (x) &\le \| \ell_\nu \|_\infty + c_1 \int_0^{t} \bE [ | \nabla p_{t-s} (X^n_s - x) | ] \diff s \\
&\eqcolon \| \ell_\nu \|_\infty + c_1 \int_{0}^{t} E_s \diff s. \label{sup_bound2:eq0} 
\end{align}

\textbf{Step 1:} We are going to prove the existence of $m \in [\![1, n]\!]$ such that $\frac{m}{n}$ depends only on $\Theta_1$ and  that $\| \ell^n_{t_k} \|_\infty \le 2 \|\ell_\nu\|_\infty$ for every $k \in [\![0, m ]\!]$. The idea is to pick $m$ such that the following induction argument is valid. We fix $k \in [\![1, m ]\!]$ and assume that $\|\ell^n_{t_i}\|_\infty \le 2 \|\ell_\nu\|_\infty$ for every $i \in [\![0, k-1 ]\!]$. Let $s \in (t_i, t_{i+1})$ for some $i \in [\![0, k-1 ]\!]$. We have
\begin{align}
E_s &\lesssim \int_{\bR^d \times \bR^d} \ell^n_{t_i} (y) p_{2(s-t_i)} (y-z) | \nabla p_{t_k-s} (x-z) | \diff y \diff z \qtext{by \eqref{sup_bound1:ineq1}} \\
&\lesssim \frac{1}{\sqrt{t_k-s}} \int_{\bR^d \times \bR^d} \ell^n_{t_i} (y) p_{2(s-t_i)} (y-z) p_{t_k-s} (x-z) \diff y \diff z \qtext{by \eqref{gaussian_heat_kernel:ineq3}} \\
&\lesssim \frac{1}{\sqrt{t_k-s}} \int_{\bR^d \times \bR^d} \ell^n_{t_i} (y) p_{2(s-t_i)} (y-z) p_{2(t_k-s)} (x-z) \diff z \diff y \qtext{by \eqref{gaussian_heat_kernel:ineq2}} \\
&= \frac{1}{\sqrt{t_k-s}} \int_{\bR^d} \ell^n_{t_i} (y) p_{2(t_k-t_i)} (x-y) \diff y \qtext{by Chapman–Kolmogorov equation} \\
&\le \frac{2 \|\ell_\nu\|_\infty}{\sqrt{t_k-s}} \qtext{by inductive hypothesis}. \label{density_sup_bound2:ineq2a}
\end{align}

By \eqref{sup_bound2:eq0} and \eqref{density_sup_bound2:ineq2a}, there exists a constant $c_2 >0$ (depending on $\Theta_1$) such that
\begin{align}
\ell^n_{t_k} (x) &\le \|\ell_\nu\|_\infty \bigg [1 + c_1c_2 \int_0^{t_k} \frac{\diff s}{\sqrt{t_k-s}} \bigg ] \\
&= \|\ell_\nu\|_\infty ( 1 + 2c_1c_2 \sqrt{t_k} ).
\end{align}

It suffices to choose $m$ such that
\begin{align} \label{density_sup_bound2:ineq5}
c_1c_2 \sqrt{t_k} \le \frac{1}{2}.
\end{align}

Clearly, there exists a constant $c_3 \in (0, 1)$ (depending on $\Theta_1$) such that if $\frac{m}{n} \le c_3$ then \eqref{density_sup_bound2:ineq5} holds.

\textbf{Step 2:} Repeating \textbf{Step 1} at most $\lceil 1/c_3 \rceil$ times, we have $\| \ell^n_{t_k} \|_\infty \le 2^{\lceil 1/c_3 \rceil} \|\ell_\nu\|_\infty$ for every $k \in [\![0, n]\!]$.

\textbf{Step 3:} By \eqref{sup_bound1:ineq2}, there exists a constant $c_4 \ge 1$ (depending on $\Theta_1$) such that if $t \in (t_k, t_{k+1})$ then $\| \ell^n_{t} \|_\infty \le c_4 \| \ell^n_{t_k} \|_\infty$. Thus
\begin{equation} \label{density_sup_bound2:ineq6}
\sup_{t \in \bT} \| \ell^n_{t} \|_\infty \le c_4 2^{\lceil 1/c_3 \rceil} \|\ell_\nu\|_\infty.
\end{equation}

\subsection{Hölder continuity in space} \label{holder:space}

We fix $x, x' \in \bR^d$. By \zcref{duhamel} and \ref{main_assmpt1:1},
\begin{align}
|\ell^n_{t} (x) - \ell^n_{t} (x')| &\lesssim \begin{myaligned}[t]
& | P_t \ell_\nu (x) - P_t \ell_\nu (x') | \\
& + \int_0^t \bE [| \nabla p_{t-s} (X^n_s- x) - \nabla p_{t-s} (X^n_s- x') |] \diff s
\end{myaligned} \\
&\eqcolon I_1 + \int_0^t I^n_2 (s) \diff s. \label{holder:space:eq1}
\end{align}

We consider the Gaussian process governed by the SDE
\[
\diff Y^{x}_{t} = \sqrt 2 \diff B_t,
\quad t \in \bT, Y^{x}_0 = x.
\]

By \cite[Theorem 1.1(2)]{huang2022singular},
\begin{equation} \label{holder:space:ineq1}
\bE \big [ \sup_{t \in \bT} | Y^{x}_{t} - Y^{x'}_{t} | \big ] \lesssim |x-x'|.
\end{equation}

On the other hand,
\begin{align}
I_1 &= | \bE [\ell_\nu (Y^{x}_{t})] - \bE [\ell_\nu (Y^{x'}_{t})] | \\
&\le [\ell_\nu]_{\alpha} \bE [ | Y^{x}_{t} - Y^{x'}_{t} |^\alpha ] \\
&\le [\ell_\nu]_{\alpha} ( \bE [ | Y^{x}_{t} - Y^{x'}_{t} | ] )^\alpha \qtext{by Jensen's inequality}\\
&\lesssim [\ell_\nu]_{\alpha} |x-x'|^\alpha \qtext{by \eqref{holder:space:ineq1}}. \label{holder:space:ineq2}
\end{align}

Next we bound $I^n_2 (s)$. We have
\begin{align}
I^n_2 (s) &\lesssim s^{-\frac{1+\alpha}{2}} |x-x'|^\alpha \bE [ p_{4(t-s)} (X^n_s-x) + p_{4(t-s)} (X^n_s-x') ] \qtext{by \eqref{gaussian_heat_kernel:ineq4}} \\
&= s^{-\frac{1+\alpha}{2}} |x-x'|^\alpha \int_{\bR^d} \ell^n_s (y) \{p_{4(t-s)}(y-x) + p_{4(t-s)}(y-x')\} \diff y \\
&\lesssim \|\ell_\nu\|_\infty s^{-\frac{1+\alpha}{2}} |x-x'|^\alpha \qtext{by \eqref{density_sup_bound2:ineq6}}. \label{holder:space:ineq4}
\end{align}

By \eqref{holder:space:eq1}, \eqref{holder:space:ineq2} and \eqref{holder:space:ineq4},
\begin{align}
\frac{|\ell^n_{t} (x) - \ell^n_{t} (x')|}{|x-x'|^\alpha} &\lesssim [\ell_\nu]_{\alpha} + \|\ell_\nu\|_\infty \int_0^t s^{-\frac{1+\alpha}{2}} \diff s \\
&\lesssim \| \ell_\nu \|_{C^\alpha_b} \qtext{because $\alpha \in (0, 1)$}. \label{holder:space:ineq5}
\end{align}

Then \eqref{density_sup_bound2:ineq6} and \eqref{holder:space:ineq5} imply \eqref{main_thm1:ineq1}.

\subsection{Hölder continuity in time}

We fix $x \in \bR^d$ and $0 \le s < t \le T$.

\textit{Step 1: we will prove \eqref{main_thm1:ineq2}.} By \zcref{duhamel},
\begin{equation} \label{holder:time:eq1}
\begin{split}
\ell^n_t (x) &= P_{t} \ell_\nu (x) + \int_0^t \bE [ \langle b^n (r, X^n_{\tau^n_r}), \nabla p_{t-r} (X^n_r - x) \rangle ] \diff r, \\
\ell^n_s (x) &= P_{s} \ell_\nu (x) + \int_0^s \bE [ \langle b^n (r, X^n_{\tau^n_r}), \nabla p_{s-r} (X^n_r - x) \rangle ] \diff r.
\end{split}
\end{equation}

By \eqref{holder:time:eq1} and \ref{main_assmpt1:1},
\begin{align}
|\ell^n_t (x) - \ell^n_s (x)| &\lesssim \begin{myaligned}[t]
& |P_{t} \ell_\nu (x) - P_{s} \ell_\nu (x)| \\
& + \int_0^s \bE [ | \nabla p_{t-r} (X^n_r - x) - \nabla p_{s-r} (X^n_r - x) | ] \diff r \\
& + \int_s^t \bE [ | \nabla p_{t-r} (X^n_r - x) | ] \diff r
\end{myaligned} \\
&\eqcolon I_1 (x) + I^n_2 (x) + I^n_3 (x). \label{holder:time:eq1a}
\end{align}

It holds for every $i \in \{0, 1\}$ that
\begin{align}
& \nabla^i p_{t-r} (y - x) - \nabla^i p_{s-r} (y - x) \\
= & \nabla^i_y \int_{\bR^d} p_{s-r} (y- z) p_{t-s} (z-x) \diff z - \nabla^i p_{s-r} (y- x) \label{holder:time:eq2} \\
= & \nabla^i_y \int_{\bR^d} p_{s-r} (y- z) p_{t-s} (z-x) \diff z - \nabla^i_y \int_{\bR^d} p_{s-r} (y- x) p_{t-s} (z- x) \diff z \\
= & \int_{\bR^d} \{ \nabla^i p_{s-r} (y- z) - \nabla^i p_{s-r} (y- x) \} p_{t-s} (z- x) \diff z. \label{holder:time:eq3}
\end{align}

First,
\begin{align}
I_1 (x) &= \bigg | \int_{\bR^d} \ell_\nu(y) \{p_{t} (y-x) - p_{s} (y-x) \} \diff y \bigg | \\
&= \bigg | \int_{\bR^d} \bigg [ \int_{\bR^d} \ell_\nu (y) \{ p_{s} (y-z) - p_{s} (y-x) \} \diff y \bigg ] p_{t-s} (z - x) \diff z \bigg | \qtext{by \eqref{holder:time:eq3}} \\
&\le \int_{\bR^d} \bigg | \int_{\bR^d} \ell_\nu (y) \{ p_{s} (y- z) - p_{s} (y - x) \} \diff y \bigg | p_{t-s} (z - x) \diff z \\
&= \int_{\bR^d} \bigg | \int_{\bR^d} p_s (y) \{ \ell_\nu (y+ z) - \ell_\nu (y + x) \} \diff y \bigg | p_{t-s} (z - x) \diff z \\
&\preccurlyeq \int_{\bR^d} |z-x|^{\alpha} p_{t-s} (z - x) \diff z  \qtext{because $\ell_\nu \in C^\alpha_b (\bR^d)$} \\
&\lesssim (t-s)^{\frac{\alpha}{2}}. \label{holder:time:ineq2}
\end{align}

Second,
\begin{align}
I^n_2 (x) &= \int_0^s \bE \bigg [ \bigg | \int_{\bR^d} \{ \nabla p_{s-r} (X^n_r- z) - \nabla p_{s-r} (X^n_r- x) \} p_{t-s} (z- x) \diff z \bigg | \bigg ] \diff r \qtext{by \eqref{holder:time:eq3}} \\
&\le \int_0^s \int_{\bR^d} \bE [ | \nabla p_{s-r} (X^n_r- z) - \nabla p_{s-r} (X^n_r- x) | ] p_{t-s} (z- x) \diff z \diff r \\
&\lesssim \begin{myaligned}[t]
& \int_0^s (s-r)^{-\frac{1+\alpha}{2}} \int_{\bR^d} |z-x|^\alpha \bE [ p_{4(s-r)} (X^n_r- z)] p_{t-s} (z- x) \diff z \diff r \\
& + \int_0^s (s-r)^{-\frac{1+\alpha}{2}} \int_{\bR^d} |z-x|^\alpha \bE [ p_{4(s-r)} (X^n_r- x)] p_{t-s} (z- x) \diff z \diff r 
\end{myaligned} \qtext{by \eqref{gaussian_heat_kernel:ineq4}}\\
&\eqcolon I^n_{21} (x) + I^n_{22} (x).
\end{align}

We have
\begin{align}
I^n_{21} (x) &= \int_0^s (s-r)^{-\frac{1+\alpha}{2}} \int_{\bR^d \times \bR^d} |z-x|^\alpha \ell^n_s (y) p_{4(s-r)} (y- z) p_{t-s} (z- x) \diff y \diff z \diff r \\
&\preccurlyeq \int_0^s (s-r)^{-\frac{1+\alpha}{2}} \int_{\bR^d \times \bR^d} |z-x|^\alpha p_{4(s-r)} (y- z) p_{t-s} (z- x) \diff y \diff z \diff r \qtext{by \eqref{density_sup_bound2:ineq6}}\\
&\lesssim \int_0^s (s-r)^{-\frac{1+\alpha}{2}} \int_{\bR^d \times \bR^d} |z-x|^\alpha p_{4(s-r)} (y- z) p_{4(t-s)} (z- x) \diff y \diff z \diff r \qtext{by \eqref{gaussian_heat_kernel:ineq2}}\\
&= \int_0^s (s-r)^{-\frac{1+\alpha}{2}} \int_{\bR^d} |z-x|^\alpha p_{4(t-s)} (z- x) \diff z \diff r \\
&\lesssim (t-s)^{\frac{\alpha}{2}} \int_0^s (s-r)^{-\frac{1+\alpha}{2}} \diff r. \label{holder:time:ineq3}
\end{align}

Similarly,
\begin{equation}
I^n_{22} (x) \preccurlyeq (t-s)^{\frac{\alpha}{2}} \int_0^s (s-r)^{-\frac{1+\alpha}{2}} \diff r. \label{holder:time:ineq4}
\end{equation}

Third,
\begin{align}
I^n_3 (x) &= \int_s^t \int_{\bR^d} \ell^n_r (y) | \nabla p_{t-r} (y - x) | \diff y \diff r \\ 
&\lesssim \int_s^t \frac{1}{\sqrt{t-r}} \int_{\bR^d} \ell^n_r (y) p_{t-r} (y - x) \diff y \diff r \qtext{by \eqref{gaussian_heat_kernel:ineq3}} \label{holder:time:ineq4a}\\ 
&\preccurlyeq \int_s^t \frac{1}{\sqrt{t-r}} \int_{\bR^d} p_{t-r} (y - x) \diff y \diff r \qtext{by \eqref{density_sup_bound2:ineq6}} \\ 
&= \int_s^t \frac{\diff r}{\sqrt{t-r}} \lesssim  \sqrt{t-s} \lesssim (t-s)^{\frac{\alpha}{2}}. \label{holder:time:ineq5}
\end{align}

By \eqref{holder:time:ineq2}, \eqref{holder:time:ineq3}, \eqref{holder:time:ineq4} and \eqref{holder:time:ineq5},
\begin{align}
|\ell^n_t (x) - \ell^n_s (x)| &\preccurlyeq (t-s)^{\frac{\alpha}{2}} \bigg [1+ \sup_{s \in \bT} \int_0^s (s-r)^{-\frac{1+\alpha}{2}} \diff r\bigg ] \\
&\lesssim (t-s)^{\frac{\alpha}{2}} \qtext{because $\alpha \in (0, 1)$} .
\end{align}

This implies \eqref{main_thm1:ineq2}.

\textit{Step 2: we will prove \eqref{main_thm1:ineq4}.} We have
\begin{align}
I_1 (x) &\le \int_{\bR^d} \ell_\nu(y) | p_{t} (y-x) - p_{s} (y-x) | \diff y \\
& \lesssim |t-s|^{\frac{\alpha}{2}} \int_{\bR^d} \ell_\nu(y) \{ t^{-\frac{\alpha}{2}} p_{2t} (y-x) + s^{-\frac{\alpha}{2}} p_{2s} (y-x) \} \diff y \qtext{by \eqref{gaussian_heat_kernel:ineq5}} \\
& \lesssim (t-s)^{\frac{\alpha}{2}} s^{-\frac{\alpha}{2}} \int_{\bR^d} \ell_\nu(y) \{ p_{2t} (y-x) + p_{2s} (y-x) \} \diff y .
\end{align}

Then
\begin{align}
& \int_{\bR^d} (1 + |x|^p) I_1 (x) \diff x \\
\lesssim  & (t-s)^{\frac{\alpha}{2}} s^{-\frac{\alpha}{2}} \int_{\bR^d} \ell_\nu(y) \bigg [ \int_{\bR^d} (1 + |x|^p)  \{ p_{2t} (y-x) + p_{2s} (y-x) \} \diff x  \bigg ] \diff y \\
\lesssim  & (t-s)^{\frac{\alpha}{2}} s^{-\frac{\alpha}{2}} \int_{\bR^d} (1 + |y|^p) \ell_\nu(y) \diff y \\
\preccurlyeq  & (t-s)^{\frac{\alpha}{2}} s^{-\frac{\alpha}{2}}. \label{weighted-norm:ineq1}
\end{align}

We have
\begin{align}
I^n_2 (x) & \lesssim \begin{multlined}[t]
|t-s|^{\frac{\alpha}{2}} \int_0^s \bE [ (t-r)^{-\frac{1+\alpha}{2}} p_{2(t-r)} (X^n_r - x) \\
+ (s-r)^{-\frac{1+\alpha}{2}} p_{2(s-r)} (X^n_r - x) ] \diff r
\end{multlined} \qtext{by \eqref{gaussian_heat_kernel:ineq5}} \\
& \lesssim |t-s|^{\frac{\alpha}{2}} \int_0^s (s-r)^{-\frac{1+\alpha}{2}} \bE [ p_{2(t-r)} (X^n_r - x) ] \diff r .
\end{align}

Then
\begin{align}
& \int_{\bR^d} (1 + |x|^p) I^n_2 (x) \diff x \\
\lesssim  & |t-s|^{\frac{\alpha}{2}} \int_0^s (s-r)^{-\frac{1+\alpha}{2}} \bE \bigg [ \int_{\bR^d} (1+|x|^p) p_{2(t-r)} (X^n_r - x) \diff x \bigg ] \diff r \\
\lesssim  & |t-s|^{\frac{\alpha}{2}} \int_0^s (s-r)^{-\frac{1+\alpha}{2}} \bE [ 1+|X^n_r|^p ] \diff r \\
\preccurlyeq  & |t-s|^{\frac{\alpha}{2}} \int_0^s (s-r)^{-\frac{1+\alpha}{2}} \diff r \qtext{by \eqref{main_thm1:ineq3}} \\
\lesssim  & |t-s|^{\frac{\alpha}{2}} \qtext{because $\alpha \in (0, 1)$}. \label{weighted-norm:ineq2}
\end{align}

We have
\begin{align}
& \int_{\bR^d} (1 + |x|^p) I^n_3 (x) \diff x \\
\lesssim  & \int_s^t \frac{1}{\sqrt{t-r}} \int_{\bR^d} \ell^n_r (y) \bigg [ \int_{\bR^d} (1 + |x|^p) p_{t-r} (y - x) \diff x \bigg ] \diff y \diff r \qtext{by \eqref{holder:time:ineq4a}} \\
\lesssim  & \int_s^t \frac{1}{\sqrt{t-r}}\bigg [  \int_{\bR^d} (1 + |y|^p) \ell^n_r (y) \diff y \bigg ] \diff r \\
\preccurlyeq  & \int_s^t \frac{\diff r}{\sqrt{t-r}} \qtext{by \eqref{main_thm1:ineq3}} \\
\lesssim  & \sqrt{t-s} \lesssim (t-s)^{\frac{\alpha}{2}}. \label{weighted-norm:ineq3}
\end{align}

We have \eqref{holder:time:eq1a} together with \eqref{weighted-norm:ineq1}, \eqref{weighted-norm:ineq2} and \eqref{weighted-norm:ineq3} implies \eqref{main_thm1:ineq4}.

\textit{Step 3: we will prove \eqref{main_thm1:ineq5}.} We assume, in addition, that $\int_{\bR^d} (1+|x|^p) \sqrt{\ell_\nu (x)} \diff x  \le C$ holds. Let $Y$ be a random variable whose density is $p_1$. Then
\begin{align}
I_1 (x) & = | \bE [ \ell_\nu (x + \sqrt{t} Y) - \ell_\nu (x + \sqrt{s} Y) ] | \\
&  \le \bE [ \{ | \ell_\nu (x + \sqrt{t} Y) - \ell_\nu (x + \sqrt{s} Y) |^2 \}^{\frac{1}{2}} ] \\
& \preccurlyeq \bE [ | (\sqrt{t} - \sqrt{s}) Y|^{\frac{\alpha}{2}} \times | \ell_\nu (x + \sqrt{t} Y) - \ell_\nu (x + \sqrt{s} Y) |^{\frac{1}{2}} ] \qtext{because $\ell_\nu \in C^\alpha_b (\bR^d)$} \\
& \lesssim (t-s)^{\frac{\alpha}{4}} \bE [ | Y |^{\frac{\alpha}{2}} \{ | \ell_\nu (x + \sqrt{t} Y) |^{\frac{1}{2}} + | \ell_\nu (x + \sqrt{s} Y) |^{\frac{1}{2}} \} ] .
\end{align}

Thus
\begin{align}
& \int_{\bR^d} (1 + |x|^p) I_1 (x) \diff x \\
\preccurlyeq  & (t-s)^{\frac{\alpha}{4}} \bE [ | Y |^{\frac{\alpha}{2}} \int_{\bR^d} (1 + |x|^p) \{ | \ell_\nu (x + \sqrt{t} Y) |^{\frac{1}{2}} + | \ell_\nu (x + \sqrt{s} Y) |^{\frac{1}{2}} \} \diff x ] \\
\lesssim  & (t-s)^{\frac{\alpha}{4}} \bE [ | Y |^{\frac{\alpha}{2}} \int_{\bR^d} (1 + |x|^p + |Y|^p) \sqrt{\ell_\nu (x)} \diff x ] \\
\preccurlyeq  & (t-s)^{\frac{\alpha}{4}} \bE [ | Y |^{\frac{\alpha}{2}} (1 + |Y|^p ) ] \\
\le  & (t-s)^{\frac{\alpha}{4}}. \label{weighted-norm:ineq4}
\end{align}

We have \eqref{holder:time:eq1a} together with \eqref{weighted-norm:ineq4}, \eqref{weighted-norm:ineq2} and \eqref{weighted-norm:ineq3} implies \eqref{main_thm1:ineq5}. This completes the proof.

\section{Existence of a weak solution} \label{proof:main_thm2}

This section is dedicated to the proof of \zcref{main_thm2}.

\subsection{Convergence of marginal densities}

By \eqref{el_sch} and \ref{main_assmpt1:1},
\begin{equation} \label{main_thm2:ineq0}
\sup_{n \in \bN} \sup_{t \in \bT} M_{p + \alpha} (\mu^n_t) \lesssim 1+ M_{p + \alpha} (\nu).
\end{equation}

We denote $B^c_R \coloneq \bR^d \setminus B(0, R)$ for every $R>0$. Then we have a uniform control of tail moment:
\begin{align}
& \sup_{n \in \bN} \sup_{t \in \bT} \int_{B_R^c} |x|^p \diff \mu^n_t (x) \\
\le  & \frac{1}{R} \sup_{n \in \bN} \sup_{t \in \bT} \int_{\bR^d} |x|^{p+\alpha} \diff \mu^n_t (x) \qtext{by Markov's inequality} \\
\lesssim  & \frac{1 + M_{p + \alpha} (\nu)}{R} \qtext{by \eqref{main_thm2:ineq0}} \\
\preccurlyeq  & \frac{1}{R} \qtext{by \ref{main_assmpt1:2}.} \label{main_thm2:ineq1}
\end{align}

Similarly,
\begin{equation} \label{main_thm2:ineq2}
\sup_{n \in \bN} \sup_{t \in \bT} \mu^n_t (B_R^c) \preccurlyeq \frac{1}{R}.
\end{equation}

By \zcref{main_thm1} and Arzelà–Ascoli theorem, there exist a sub-sequence (also denoted by $(\ell^n)$ for simplicity) and a continuous function $\ell : \bT \times \bR^d \to \bR_+$ such that
\begin{equation} \label{main_thm2:eq1}
\lim_{n} \sup_{t \in \bT} \sup_{x \in B(0, R)} |\ell^n_t (x) - \ell_t (x)| =0
\qtextq{for} R>0 .
\end{equation}

Above, $\ell_t \coloneq \ell (t, \cdot)$. Clearly, $\ell_0 = \ell_\nu$ and
\begin{align}
\sup_{t \in \bT} \| \ell_t \|_{C^{\alpha}_b} &\lesssim \| \ell_\nu \|_{C^{\alpha}_b}, \label{main_thm2:ineq3} \\
\| \ell_t - \ell_s \|_\infty &\preccurlyeq |t-s|^{\frac{\alpha}{2}} \qtextq{for} s, t \in \bT. \label{main_thm2:ineq4}
\end{align}

Using above estimates together with the same reasoning as in \cite[Section 4.2]{le2024wellposedness}, we get that $\ell_t$ is a density whose induced distribution $\mu_t \in \sP_p (\bR^d)$; and that
\begin{align} 
\limsup_n \sup_{t \in \bT} W_p (\mu^n_t, \mu_t) &= 0, \label{main_thm2:eq2} \\
W_p (\mu_t, \mu_s) &\lesssim |t-s|^{\frac{1}{2}} \label{main_thm2:eq2a}, \\
\sup_{t \in \bT} \int_{B_R^c} |x|^p \diff \mu_t (x) &\preccurlyeq \frac{1}{R}, \label{main_thm2:ineq4a} \\
\sup_{t \in \bT} \mu_t (B_R^c) &\preccurlyeq \frac{1}{R} \qtextq{for} R >0. \label{main_thm2:ineq4b}
\end{align}

\subsection{Existence of a weak solution}

We fix $(f, g) \in C^{\infty}_c (0, T) \times C^{\infty}_c (\bR^d)$. We apply Itô's lemma to \eqref{main_eq2} and the map $(t, x) \mapsto f(t) g(x)$. Then
\begin{align} \label{main_thm2:eq3}
\bE \big [ \int_0^t f'(s) g(X^n_s) \diff s \big ] &= \begin{multlined}[t]
- \bE \big [ \int_0^t \langle f(s) \nabla g (X^n_s), b^n(s, X^n_{\tau^n_s}) \rangle \diff s \big ] \\
- \bE \big [ \int_0^t f(s) \Delta g (X^n_s) \diff s \big ]. 
\end{multlined}
\end{align}

Let
\begin{align}
I_1^n (s) &\coloneq \bE [ g(X^n_s) ], \\
I_3^n (s) &\coloneq \bE [ \Delta g (X^n_s) ] , \\
I^n_2 (s) &\coloneq \bE [ \langle \nabla g (X^n_s), b^n (s, X^n_{\tau^n_s}) \rangle ] \\
&= \begin{myaligned}[t]
& \int_{\bR^d \times \bR^d} \langle \nabla g (y), b^n (s, x) \rangle \ell^n_{\tau^n_s} (x) \\
& \times p_{s-\tau^n_s} \bigg (x-y + \int_{\tau^n_s}^s b^n_r (x) \diff r \bigg ) \diff x \diff y.
\end{myaligned} \qtext{by \eqref{sup_bound1:eq1}} \label{main_thm2:eq4}
\end{align}

By \eqref{main_thm2:eq3} and Fubini's theorem,
\begin{equation} \label{main_thm2:eq4a}
\int_0^t f'(s) I_1^n (s) \diff s = - \int_0^t f(s) I^n_2 (s) \diff s - \int_0^t f(s) I_3^n (s) \diff s.
\end{equation}

By \eqref{main_thm2:eq2},
\begin{equation} \label{main_thm2:eq5}
\begin{split}
\lim_n I^n_1 (s) &= \int_{\bR^d} g(x) \ell(s, x) \diff x, \\
\lim_n I^n_3 (s) &= \int_{\bR^d} \Delta g(x) \ell(s, x) \diff x.
\end{split}
\end{equation}

Next we consider $I^n_2 (s)$. It holds for $n$ large enough that $\supp f \subset (\eps_n, T)$. WLOG, we assume $s \in (\eps_n, T)$ and thus $b^n (s, x) = b (s, x, \ell^n_{\tau^n_s} (x), \mu^n_{\tau^n_s})$. For every $x \in \bR^d$, we define $\nu^{n, x} \in \sP (\bR^d)$, $h^n (x) \in \bR$ and $I_2 (s) \in \bR$ by
\begin{align}
\diff \nu^{n, x} (y) &\coloneq p_{s-\tau^n_s} \bigg (x-y + \int_{\tau^n_s}^s b^n (r, x) \diff r \bigg ) \diff y , \label{main_thm2:eq6} \\
h^n (x) &\coloneq \int_{\bR^d} \langle \nabla g (y), b (s, x, \ell^n_{\tau^n_s} (x), \mu^n_{\tau^n_s}) \rangle \diff \nu^{n, x} (y), \label{main_thm2:eq7} \\
I_2 (s) &\coloneq \int_{\bR^d} \ell_{s} (x) \langle \nabla g (x), b (s, x, \ell_{s} (x), \mu_{s}) \rangle \diff x. \label{main_thm2:eq8}
\end{align}

Then
\begin{align}
\int_{\bR^d} | x-y | \diff \nu^{n, x} (y) &= \int_{\bR^d} | y | p_{s-\tau^n_s} \bigg (y + \int_{\tau^n_s}^s b^n (r, x) \diff r \bigg ) \diff y , \\
& \lesssim \sqrt{s-\tau^n_s} \qtext{by \ref{main_assmpt1:1} and \eqref{gaussian_heat_kernel:ineq1}} , \label{main_thm2:ineq5} \\
I^n_2 (s) &= \int_{\bR^d} \ell^n_{\tau^n_s} (x) h^n (x) \diff x \qtext{by \eqref{main_thm2:eq4}, \eqref{main_thm2:eq6} and \eqref{main_thm2:eq7}}. \label{main_thm2:eq9}
\end{align}

WLOG, we assume $\| f \|_\infty + \| \nabla g \|_\infty + \| \nabla^2 g \|_\infty \le 1$. By \eqref{main_thm2:eq8} and \eqref{main_thm2:eq9},
\begin{align} \label{main_thm2:ineq6}
| I^n_2 (s) - I_2 (s) | &\lesssim \begin{myaligned}[t]
& \int_{\bR^d} |\ell^n_{\tau^n_s} (x) - \ell_{s} (x)| \times | \langle \nabla g (x), b (s, x, \ell_{s} (x), \mu_{s}) \rangle | \diff x \\
& + \int_{\bR^d} \ell^n_{\tau^n_s} (x) |h^n (x) - \langle \nabla g (x), b (s, x, \ell_{s} (x), \mu_{s}) \rangle | \diff x.
\end{myaligned}
\end{align}

First,
\begin{align}
& |h^n (x) - \langle \nabla g (x), b (s, x, \ell_{s} (x), \mu_{s}) \rangle | \\
\lesssim  &\begin{myaligned}[t]
& \int_{\bR^d} | \nabla g (y) - \nabla g (x) | \diff \nu^{n, x} (y) \\
& + | b (s, x, \ell^n_{\tau^n_s} (x), \mu^n_{\tau^n_s}) - b (s, x, \ell_{s} (x), \mu_{s})| 
\end{myaligned} \qtext{by \eqref{main_thm2:eq7}}\\
\lesssim  & \int_{\bR^d} | x-y | \diff \nu^{n, x} (y) + | \ell^n_{\tau^n_s} (x) - \ell_{s} (x) | + W_p (\mu^n_{\tau^n_s}, \mu_{s}) \qtext{by \ref{main_assmpt1:3}} \\
\lesssim  & \sqrt{s-\tau^n_s} + | \ell^n_{\tau^n_s} (x) - \ell_{s} (x) | + W_p (\mu^n_{\tau^n_s}, \mu_{s}) \qtext{by \eqref{main_thm2:ineq5}} . \label{main_thm2:ineq7}
\end{align}

Let $S \coloneq \supp g$. Then $S$ is compact. By \eqref{main_thm2:ineq6} and \eqref{main_thm2:ineq7},
\begin{align} \label{main_thm2:ineq8}
| I^n_2 (s) - I_2 (s) | &\lesssim \begin{myaligned}[t]
& \int_{S} |\ell^n_{\tau^n_s} (x) - \ell_{s} (x)| \diff x + \sqrt{s-\tau^n_s} \\
& + \int_{\bR^d} \ell^n_{\tau^n_s} (x) | \ell^n_{\tau^n_s} (x) - \ell_{s} (x) | \diff x + W_p (\mu^n_{\tau^n_s}, \mu_{s}).
\end{myaligned}
\end{align}

By \eqref{main_thm1:ineq2} and \eqref{main_thm2:eq1}, it holds for every $R>0$ that
\begin{equation} \label{main_thm2:eq10}
\lim_{n} \sup_{x \in B(0, R)} |\ell^n_{\tau^n_s} (x) - \ell_s (x)| =0 .
\end{equation}

By \eqref{main_thm1:ineq3} and \eqref{main_thm2:eq2},
\begin{equation} \label{main_thm2:eq11}
\lim_n W_p (\mu^n_{\tau^n_s}, \mu_s) =0.
\end{equation}

We have
\begin{align}
& \limsup_n | I^n_2 (s) - I_2 (s) | \\
\lesssim  & \limsup_n \int_{\bR^d} \ell^n_{\tau^n_s} (x) | \ell^n_{\tau^n_s} (x) - \ell_{s} (x) | \diff x \qtext{by \eqref{main_thm2:ineq8}, \eqref{main_thm2:eq10} and \eqref{main_thm2:eq11}}\\
\preccurlyeq  & \limsup_n \int_{\bR^d} | \ell^n_{\tau^n_s} (x) - \ell_{s} (x) | \diff x \qtext{by \eqref{main_thm1:ineq1}}\\
\le  & \limsup_n \int_{B(0, R)} | \ell^n_{\tau^n_s} (x) - \ell_{s} (x) | \diff x + \limsup_n \int_{B_R^c} | \ell^n_{\tau^n_s} (x) - \ell_{s} (x) | \diff x \\
\eqcolon  & \limsup_n I^n_{41} (s, R) + \limsup_n I^n_{42} (s, R) \qtextq{for every} R>0.
\end{align}

By \eqref{main_thm2:eq10}, $\limsup_n I^n_{41} (s, R) =0$. By \eqref{main_thm2:ineq2} and \eqref{main_thm2:ineq4b}, $\limsup_n I^n_{42} (s, R) \preccurlyeq \frac{1}{R}$. Then $\limsup_n | I^n_2 (s) - I_2 (s) |  \preccurlyeq \frac{1}{R}$ and thus $\lim_n I^n_2 (s) = I_2 (s)$. This together with \eqref{main_thm2:eq4a}, \eqref{main_thm2:eq5} and dominated convergence theorem implies
\begin{align}
\int_0^t \int_{\bR^d} f'(s) g(x) \ell(s, x) \diff x \diff s &= \begin{multlined}[t]
- \int_0^t \int_{\bR^d} f(s) \langle \nabla g (x), b (s, x, \ell (s, x), \mu_{s}) \ell (s, x) \rangle \diff x \diff s \\
- \int_0^t \int_{\bR^d} f(s) \Delta g(x) \ell(s, x) \diff x \diff s.
\end{multlined}
\end{align}

Hence $\ell$ satisfies the following Fokker-Planck PDE in distributional sense
\begin{align}
\partial_t \ell (t, x) = - \sum_{i=1}^d \partial_{x_i} \{b^i (t, x, \ell (t, x), \mu_{t}) \ell (t, x)\} + \Delta  \ell (t, x).
\end{align}

Clearly,
\begin{enumerate}
\item $(t, x) \mapsto b(t, x, \ell (t, x), \mu_t)$ is measurable with
\[
\int_{\bT} \int_{\bR^d} |b(t, x, \ell (t, x), \mu_t)| \diff \mu_t (x) \diff t < \infty.
\]
\item $\bT \to \sP_p (\bR^d), \, t \mapsto \mu_t$ is continuous by \eqref{main_thm2:eq2a}.
\end{enumerate}

We apply superposition principle as in \cite[Section 2]{barbu2020nonlinear} and get that \eqref{main_eq1} has a weak solution whose marginal distribution is exactly $(\mu_t, t \in \bT)$. This completes the proof.

\section{Rate of convergence} \label{proof:main_thm3}

This section is dedicated to the proof of \zcref{main_thm3}.

\subsection{Decomposition of error}

For every $(s, z) \in \bT \times \bR^d$, we define
\begin{align}
I^n_1 (s, z) & \coloneq \bE [ \langle b^n (s, X^n_{\tau^n_s}) , \nabla p_{t-s} (X^n_s - z) - \nabla p_{t-s} (X^n_{\tau^n_s} - z) \rangle], \\
I^n_2 (s, z) & \coloneq \bE [ \langle b^n (s, X^n_{\tau^n_s}) , \nabla p_{t-s} (X^n_{\tau^n_s} - z) \rangle] - \bE [\langle b^n (s, X^n_{s}) , \nabla p_{t-s} (X^n_{s} - z) \rangle] , \\
I^n_3 (s, z) & \coloneq 1_{(\eps_n, T]} (s) \bE [ \langle  b (s, X^n_{s}, \ell^n_{\tau^n_s} (X^n_{s}), \mu^n_{\tau^n_s}) - b (s, X^n_s, \ell^n_s (X^n_s), \mu^n_s), \nabla p_{t-s} (X^n_s - z) \rangle ] , \\
I^n_4 (s, z) & \coloneq 1_{[0, \eps_n]} (s) \bE [\langle b (s, X^n_s, \ell^n_s (X^n_s), \mu^n_s) , \nabla p_{t-s} (X^n_{s} - z) \rangle] .
\end{align}

We also define the density-like function $\hat \ell^n_t : \bR^d \to \bR$ by
\[
\hat \ell^n_t (z) \coloneq P_{t} \ell_\nu (z) + \int_0^t \bE [ \langle b (s, X^n_s, \ell^n_s (X^n_s), \mu^n_s), \nabla p_{t-s} (X^n_s- z) \rangle ] \diff s .
\]

By \zcref{duhamel},
\[
\ell^n_t (z) - \hat \ell^n_t (z) = \int_0^t \{ I^n_1 (s, z) + I^n_2 (s, z) + I^n_3 (s, z) - I^n_4 (s, z) \} \diff s.
\]

We define $f, \hat f : \bT \to \bR_+$ by
\begin{align}
f(t) &\coloneq \int_{\bR^d} ( 1+|z|^p) |\ell^n_t (z) - \ell_t (z)| \diff z , \\
\hat f(t) &\coloneq \int_{\bR^d} ( 1+|z|^p) |\hat \ell^n_t (z) - \ell_t (z)| \diff z .
\end{align}

By \eqref{main_thm1:ineq3} and \eqref{main_thm2:eq2a}, $f$ is bounded. Clearly,
\begin{equation} \label{weighted-norm2:ineq1}
f(t) \le \hat f(t) + \int_0^t \int_{\bR^d} ( 1+|z|^p) \{ |I^n_1 (s, z)| + |I^n_2 (s, z)| + |I^n_3 (s, z)|  + |I^n_4 (s, z)| \} \diff z \diff s.
\end{equation}

We have
\begin{align}
|I^n_1 (s, z)|  & \lesssim \bE [ | \nabla p_{t-s} (X^n_s - z) - \nabla p_{t-s} (X^n_{\tau^n_s} - z) | ] \qtext{by \ref{main_assmpt1:1}} \\
& \lesssim \frac{1}{\sqrt{t-s}} \bE [ |X^n_s - X^n_{\tau^n_s}|^\alpha \{ p_{4(t-s)} (X^n_s - z) + p_{4(t-s)} (X^n_{\tau^n_s} - z) \} ] \qtext{by \eqref{gaussian_heat_kernel:ineq4}} \\
& \lesssim \begin{myaligned}[t]
& \frac{1}{\sqrt{t-s}} \int_{(\bR^d)^2} |x-y|^{\alpha}  \{ p_{4(t-s)} (y-z )+ p_{4(t-s)} (x -z ) \} \\
& \times \ell^n_{\tau^n_s} (x) p_{2(s-\tau^n_s)} (x-y) \diff x \diff y.
\end{myaligned} \qtext{by \eqref{sup_bound1:ineq1}}
\end{align}

There exists a constant $\kappa >0$ (depending on $\Theta_1$) such that
\begin{align}
\int_{\bR^d} (1 + |z|^p) \{ p_{4(t-s)} (y-z )+ p_{4(t-s)} (x -z ) \} \diff z &\lesssim 1+ |x|^p + |y|^p , \\
|x-y|^\alpha p_{2(s-\tau^n_s)} (x-y) & \lesssim (s-\tau^n_s)^{\frac{\alpha}{2}} p_{\kappa (s-\tau^n_s)} (x-y) .
\end{align}

Then
\begin{align}
& \int_{\bR^d} (1 + |z|^p) |I^n_1 (s, z)| \diff z \\
\lesssim  & \frac{(s-\tau^n_s)^{\frac{\alpha}{2}}}{\sqrt{t-s}} \int_{(\bR^d)^2} (1+ |x|^p + |y|^p) \ell^n_{\tau^n_s} (x) p_{\kappa (s-\tau^n_s)} (x-y) \diff x \diff y \\
\lesssim  & \frac{(s - \tau^n_s)^{\frac{\alpha}{2}}}{\sqrt{t-s}}  \int_{\bR^d} (1+ |x|^p) \ell^n_{\tau^n_s} (x) \diff x \\
\preccurlyeq  & \frac{(s - \tau^n_s)^{\frac{\alpha}{2}}}{\sqrt{t-s}} \qtext{by \eqref{main_thm1:ineq3}} . \label{weighted-norm2:ineq2}
\end{align}

We have
\begin{align}
| I^n_2 (s, z) | & \le \int_{\bR^d} |\langle b^n (s, x) , \nabla p_{t-s} (x - z) \rangle| \times |\ell^n_{\tau^n_s} (x) - \ell^n_s (x) | \diff x \\
& \lesssim \frac{1_{(\eps_n, T]} (s)}{\sqrt{t-s}} \int_{\bR^d} p_{t-s} (x-z)  |\ell^n_{\tau^n_s} (x) - \ell^n_s (x) | \diff x \qtext{by \eqref{drift}, \eqref{gaussian_heat_kernel:ineq3} and \ref{main_assmpt1:1}} .
\end{align}

Then
\begin{align}
& \int_{\bR^d} (1 + |z|^p) |I^n_2 (s, z)| \diff z \\
\lesssim  & \frac{1_{(\eps_n, T]} (s)}{\sqrt{t-s}} \int_{\bR^d} \bigg [ \int_{\bR^d} (1 + |z|^p) p_{t-s} (x-z) \diff z \bigg ] |\ell^n_{\tau^n_s} (x) - \ell^n_s (x) | \diff x \\
\lesssim  & \frac{1_{(\eps_n, T]} (s)}{\sqrt{t-s}} \int_{\bR^d} (1 + |x|^p) |\ell^n_{\tau^n_s} (x) - \ell^n_s (x) | \diff x \\
\preccurlyeq  & \frac{(s-\tau^n_s)^{\frac{\alpha}{2}} 1_{(\eps_n, T]} (s)}{|\tau^n_s|^{\frac{\alpha}{2}} \sqrt{t-s}} \qtext{by \eqref{main_thm1:ineq4}} \\
\lesssim  & \frac{(s-\tau^n_s)^{\frac{\alpha}{2}} 1_{(\eps_n, T]} (s)}{(s-\eps_n)^{\frac{\alpha}{2}} \sqrt{t-s}}. \label{weighted-norm2:ineq3}
\end{align}

We have
\begin{align}
| I^n_3 (s, z) | & \lesssim \bE [  | \nabla p_{t-s} (X^n_s - z) | ] \{ \| \ell^n_{\tau^n_s} - \ell^n_s \|_\infty + W_p (\mu^n_{\tau^n_s}, \mu^n_s) \} \qtext{by \ref{main_assmpt1:3}} \\
& \lesssim \frac{\| \ell^n_{\tau^n_s} - \ell^n_s \|_\infty + W_p (\mu^n_{\tau^n_s}, \mu^n_s)}{\sqrt{t-s}} \bE [ p_{t-s} (X^n_s - z) ] \qtext{by \eqref{gaussian_heat_kernel:ineq3}} \\
& \preccurlyeq \frac{(s - \tau^n_s)^{\frac{\alpha}{2}}}{\sqrt{t-s}} \bE [ p_{t-s} (X^n_s - z) ] \qtext{by \zcref{main_thm1}} .
\end{align}

Then
\begin{align}
\int_{\bR^d} (1 + |z|^p) | I^n_3 (s, z) | \diff z & \preccurlyeq \frac{(s - \tau^n_s)^{\frac{\alpha}{2}}}{\sqrt{t-s}} \bE \bigg [ \int_{\bR^d} (1 + |z|^p) p_{t-s} (X^n_s-z) \diff z \bigg ] \\
& \lesssim \frac{(s - \tau^n_s)^{\frac{\alpha}{2}}}{\sqrt{t-s}} \bE [1 + |X^n_s|^p ] \\
& \preccurlyeq \frac{(s - \tau^n_s)^{\frac{\alpha}{2}}}{\sqrt{t-s}} \qtext{by \eqref{main_thm1:ineq3}}. \label{weighted-norm2:ineq4}
\end{align}
By \ref{main_assmpt1:1} and \eqref{gaussian_heat_kernel:ineq3},
\begin{align}
| I^n_4 (s, z) | \lesssim  \frac{1_{[0, \eps_n]} (s)}{\sqrt{t-s}} \bE [ p_{t-s} (X^n_{s} - z) ] .
\end{align}

Then
\begin{align}
\int_{\bR^d} (1 + |z|^p) | I^n_4 (s, z) | \diff z & \lesssim \frac{1_{[0, \eps_n]} (s)}{\sqrt{t-s}} \bE \bigg [ \int_{\bR^d} (1 + |z|^p) p_{t-s} (X^n_s-z) \diff z \bigg ] \\
& \lesssim \frac{1_{[0, \eps_n]} (s)}{\sqrt{t-s}} \bE [1 + |X^n_s|^p ] \\
& \preccurlyeq \frac{1_{[0, \eps_n]} (s)}{\sqrt{t-s}} \qtext{by \eqref{main_thm1:ineq3}}. \label{weighted-norm2:ineq5}
\end{align}

\subsection{Bound weighted $L^1$ norm}

As in \cite[Section 5.1]{le2024wellposedness}, we have
\begin{align} 
\hat f(t) & \lesssim (1 + \| \ell_{\nu} \|_{\infty} + M_p (\nu) ) \int_0^t (T-s)^{-\frac{1}{2}} \{ f(s)+|f(s)|^{\frac{1}{p}} \} \diff s \\
& \preccurlyeq \int_0^t (T-s)^{-\frac{1}{2}} \{ f(s)+|f(s)|^{\frac{1}{p}} \} \diff s. \label{weighted-norm2:ineq6}
\end{align}

By \eqref{weighted-norm2:ineq1}, \eqref{weighted-norm2:ineq2}, \eqref{weighted-norm2:ineq3}, \eqref{weighted-norm2:ineq4}, \eqref{weighted-norm2:ineq5} and \eqref{weighted-norm2:ineq6},
\begin{align}
f(t) & \preccurlyeq \begin{myaligned}[t]
& (s - \tau^n_s)^{\frac{\alpha}{2}} \int_0^t \bigg \{ \frac{1}{\sqrt{t-s}} + \frac{1_{[0, \eps_n]} (s)}{(s - \tau^n_s)^{\frac{\alpha}{2}} \sqrt{t-s}} + \frac{1_{(\eps_n, T]} (s)}{(s-\eps_n)^{\frac{\alpha}{2}} \sqrt{t-s}}  \bigg \} \diff s \\
& + \int_0^t (T-s)^{-\frac{1}{2}} \{ f(s)+|f(s)|^{\frac{1}{p}} \} \diff s
\end{myaligned} \\
& \lesssim (s - \tau^n_s)^{\frac{\alpha}{2}} + \int_0^t (T-s)^{-\frac{1}{2}} \{ f(s)+|f(s)|^{\frac{1}{p}} \} \diff s.
\end{align}

Because $p=1$ and $s - \tau^n_s \le \frac{1}{n}$,
\[
f(t) \preccurlyeq n^{-\frac{\alpha}{2}} + \int_0^t (T-s)^{-\frac{1}{2}} f(s) \diff s.
\]

By Gronwall's lemma,
\[
\sup_{t \in \bT} f (t) \preccurlyeq n^{-\frac{\alpha}{2}} .
\]

This completes the proof.

\section*{Acknowledgment}

The author is grateful to Professor Stéphane Villeneuve for his guidance and encouragement, and to Professor Sébastien Gadat for his generous funding during the author's PhD.

\printbibliography

\section*{Appendix} \label{appendix}

\begin{proof} [Proof of \zcref{freezing}]
WLOG, we assume $d=1$ and $\varphi$ is non-negative. The claims that $N \in \cE$ and that $\Phi$ is measurable follows from Tonelli theorem, measurability and non-negativity of $\varphi$.
\begin{enumerate}
\item Let $\nu$ be the joint distribution of $(Y, \operatorname{id}):(\Omega, \cA) \to  (E \times \Omega, \cE \otimes \cG)$. Here $\id$ the identity map on $\Omega$. Let $\bP_\cG$ be the restriction of $\bP$ to $(\Omega, \cG)$. We have $\nu = \mu \otimes \bP_\cG$ because it holds for $A \in \cE$ and $B \in \cG$ that
\begin{align}
\nu (A \times B) &= \bP [(Y, \operatorname{id}) \in A \times B] \\
&= \bP [ \{Y \in A\} \cap \{\operatorname{id} \in B\}] \\
&= \bP [ Y \in A ] \times \bP [ \operatorname{id} \in B ] \qtext{by independence of $\cD, \cG$} \\
&= \mu ( Y \in A ) \times \bP_\cG [ B ].
\end{align}

Then
\begin{align}
\bE [ |\varphi (Y, \cdot)| ] &= \int_{E \times \Omega} |\varphi (y, \omega)| \diff \nu (y, \omega) \qtext{because $\varphi$ is measurable w.r.t $\cE \otimes \cG$}\\
&= \int_{E} \bigg ( \int_{\Omega} |\varphi (y, \omega)| \diff \bP_\cG (\omega) \bigg ) \diff \mu (y) \qtext{by Tonelli theorem} \\
&= \int_{E} \bE [|\varphi(y, \cdot)|] \diff \mu (y). \label{freezing:prf:eq1}
\end{align}

By integrability of $\varphi (Y, \cdot)$, it holds for $\mu$-a.e. $y \in E$ that $\bE [|\varphi(y, \cdot)|] < \infty$. Thus $\mu(N)=0$.

\item It remains to prove $\bE [\varphi(Y, \cdot) | \cD] = \Phi (Y)$. First, we verify that $\Phi (Y)$ is integrable. Indeed,
\begin{align}
\bE [|\Phi (Y)|] &= \int_{E} |\Phi(y)| \diff \mu (y) \\
&= \int_{E} |\bE [\varphi(y, \cdot)]| \diff \mu (y) \\
&\le \int_{E} \bE [|\varphi(y, \cdot)|] \diff \mu (y) \\
&= \bE [ |\varphi (Y, \cdot)| ] \qtext{by \eqref{freezing:prf:eq1}}\\
&< \infty. 
\end{align}

Let $D \in \cD$. We need to prove $\bE [\varphi(Y, \cdot) 1_D] = \bE[\Phi (Y)1_D]$. WLOG, we assume $\bP [D] >0$. We define a probability measure $\widetilde \bP$ on $(\Omega, \cA)$ by $\widetilde \bP [A] \coloneq \frac{\bP [A \cap D]}{\bP [D]}$ for $A \in \cA$. Let $\widetilde \bE$ be the expectation w.r.t $(\Omega, \cA, \widetilde \bP)$. Approximating with simple functions, we get $\bE [Z1_D] = \bP [D] \times \widetilde \bE [Z]$ for $Z \in L^1(\Omega, \cA, \widetilde \bP)$. So it suffices to prove $\widetilde \bE [\varphi(Y, \cdot)] = \widetilde \bE[\Phi (Y)]$. Let $\widetilde \mu$ be the distribution of $Y$ on $(E, \cE)$ under $\widetilde \bP$. Let $\widetilde \nu$ be the joint distribution of $(Y, \operatorname{id}):(\Omega, \cA) \to  (E \times \Omega, \cE \otimes \cG)$ under $\widetilde \bP$. Let $\widetilde \bP_\cG$ be the restriction of $\widetilde \bP$ to $(\Omega, \cG)$. Let's prove that $\widetilde \nu = \widetilde \mu \otimes \widetilde \bP_\cG$. Indeed, it holds for $A \in \cE$ and $B \in \cG$ that
\begin{align}
\widetilde \nu (A \times B) &= \widetilde \bP [(Y, \operatorname{id}) \in A \times B] \\
&= \widetilde \bP [ \{Y \in A\} \cap \{\operatorname{id} \in B\}] \\
&= \frac{\bP [ (\{Y \in A\} \cap D) \cap \{\operatorname{id} \in B\}]}{\bP [D]} \\
&= \frac{\bP [ \{Y \in A\} \cap D ] \times \bP[ \{\operatorname{id} \in B\} \cap D]}{(\bP [D])^2} \qtext{by independence of $\cD, \cG$ under $\bP$} \\
&= \widetilde \bP [Y \in A] \times \widetilde \bP [\operatorname{id} \in B] \\
&= \widetilde \mu ( Y \in A ) \times \widetilde \bP_\cG [ B ].
\end{align}

Then
\begin{align}
\widetilde \bE [ \varphi (Y, \cdot) ] &= \int_{E \times \Omega} \varphi (y, \omega) \diff \widetilde\nu (y, \omega) \qtext{because $\varphi$ is measurable w.r.t $\cE \otimes \cG$}\\
&= \int_{E} \bigg ( \int_{\Omega} \varphi (y, \omega) \diff \widetilde \bP_\cG (\omega) \bigg ) \diff \widetilde \mu (y) \qtext{by Fubini's theorem} \\
&= \int_{E} \widetilde \bE [\varphi(y, \cdot)] \diff \widetilde \mu (y), \\
\widetilde \bE [\varphi(y, \cdot)] &= \frac{\bE [\varphi(y, \cdot)1_D]}{\bP [D]} \\
&= \frac{\bE [\varphi(y, \cdot)] \times \bE [1_D]}{\bP [D]} \qtext{by independence of $\cD, \cG$ under $\bP$} \\
&= \bE [\varphi(y, \cdot)] = \Phi (y).
\end{align}

Then $\widetilde \bE [ \varphi (Y, \cdot) ] = \int_{E} \Phi (y) \diff \widetilde \mu (y) = \widetilde \bE [\Phi (Y)]$. This completes the proof.
\end{enumerate}
\end{proof}

\end{document}